\RequirePackage[l2tabu, orthodox]{nag}	

\documentclass[11pt]{article}

	\usepackage[utf8x]{inputenc}		

	\usepackage[OT2,T1]{fontenc}			

	\usepackage[frenchb,english]{babel}					

	\usepackage{amsmath}					
	\usepackage{amssymb}					
	\usepackage{amsthm}					

	\usepackage{graphicx}				
	\usepackage{color}					

	\usepackage{version}					

	\usepackage{geometry}				

\input{commath}

	\usepackage{hyperref}				

	\usepackage{tikz}					
	\usepackage{tikz-cd}					

	\usepackage{cleveref}				
	\usepackage{ucs}

\def\groi{{G}^{\!\!\!\raise3pt\hbox{$\scriptscriptstyle \circ$}}}

\newcounter{axioms}
\newenvironment{axioms}[1]
	{\renewcommand{\theaxioms}{{\upshape\sffamily(#1\arabic{axioms})}}
	\begin{list}{\theaxioms}
		{\usecounter{axioms}
		\settowidth{\labelwidth}{\theaxioms{} }
		\leftmargin=\labelwidth
		\advance \leftmargin\labelsep
		}
	}
	{\end{list}}

\newenvironment{axiom}[1]
	{\renewcommand{\theaxioms}{{\upshape\sffamily(#1)}}
	\begin{list}{\theaxioms}
		{\usecounter{axioms}
		\settowidth{\labelwidth}{\theaxioms{} }
		\leftmargin=\labelwidth
		\advance \leftmargin\labelsep
		}
	}
	{\end{list}}

\newtheorem*{question}{Question}

\hypersetup{%
colorlinks=true,%
linkcolor=black,%
urlcolor=black,%
citecolor=black}

\renewcommand{\thefootnote}{}

\begin{document}

\selectlanguage{english}

\begin{center}
{\Large \bfseries \scshape A linear version of Dawson-Gärtner theorem \\
Applications to Cramér's theory}\footnote{\hspace{-6mm}MSC2010 subject classifications: 60F10; 28E15.}

\vspace{5mm}

\textsc{By Pierre Petit\footnote{\hspace{-6mm}E-mail: \texttt{pierre.petit@normalesup.org}}}

\vspace{5mm}

\textit{Institut de Mathématiques de Toulouse\footnote{\hspace{-6mm}Address: UPS IMT, F-31062 Toulouse Cedex 9, France}}

\vspace{5mm}

\today
\end{center}

\vspace{5mm}

\setcounter{footnote}{0}
\renewcommand{\thefootnote}{\arabic{footnote}}

\begin{abstract}
We prove a linear version of Dawson-Gärtner theorem: weak large deviation principles and the equality $-s = p^*$ between the negentropy and the Fenchel-Legendre transform of the pressure are preserved through linear projective limits. As a result, the equality $-s = p^*$ holds in great generality for empirical means of independent and identically distributed random variables (Cramér's theory), e.g. in any measurable normed space, and even in any projective limit of such spaces. Eventually, we give an original example where $-s \neq p^*$ and discuss the dual equality.
\end{abstract}

\tableofcontents

\section{Introduction}

In 1987, Dawson and Gärtner (see \cite{DaG87}) formalized the fact that large deviation principles with good rate functions well behave with respect to projective limits. Originally designed to establish large deviation principles in weak topologies, especially in Sanov-type theorems, the result allows to extend large deviation principles to rather general contexts (see\emph{, e.g.}, \cite{DeS89}, \cite{dAc94a}, \cite{dAc94b}, \cite{DAc97}, \cite{DeZ93s}).

In the spirit of Sanov's theorem (see \cite{San57}) and Varadhan's theory of large deviations (see \cite{Var66}), via Freidlin and Wentzell's theory (see \cite{FrW84}), Dawson-Gärtner theorem relies on a compacity assumption: the goodness of the rate function. In order to handle generalizations of Cramér's theorem (see\emph{, e.g.}, \cite{Cra38}, \cite{BaZ79}, \cite{DeS89}, \cite{DeZ93s}, \cite{Cer07}), we weaken the role of compacity and exploit the role of linearity when dealing with empirical means of random variables. We prove that weak large deviation principles and the equality $-s = p^*$ between the negentropy and the Fenchel-Legendre transform of the pressure well behave with respect to linear projective limits (see \Cref{ldg}). We also prove a linear version of the inverse contraction principle (see \Cref{ii}).

In \Cref{sec:applic_cramer}, we apply the result to Cramér's theory. In \Cref{sec:cramerlcms}, we obtain a general framework in which the entropy $s$ and the pressure $p$ of a sequence of empirical means of independent and identically distributed random variables are well-defined. It can be seen as a constructive general model of Cramér's theory, closely related to the nonconstructive frameworks of \cite{BaZ79} and \cite{Cer07}, and generalizing both the setting of separable Banach spaces with their Borel $\sigma$-algebra and the setting of weak topologies with cylinder $\sigma$-algebra. It appears that \Cref{ii} and \Cref{ldg} imply the equality $-s = p^*$ in this general framework (see \Cref{cramerlcm}). In particular, the equality $-s = p^*$ is true in any separable (not necessarily complete) normed vector space with its Borel $\sigma$-algebra (see \Cref{cramerevns}), and even in any normed vector space with its Borel $\sigma$-algebra as soon as the question makes sense (see \Cref{cramerevnm}).

This analysis was motivated by a question of Rapha\"el Cerf: does there exist a sequence of empirical means of independent and identically distributed random variables for which $-s \neq p^*$? We provide such an example, on the borders of Cramér's theory, in \Cref{sec:counterexample1}. And we discuss the dual equality $p = (-s)^*$ in \Cref{sec:counterexample2}.

\section{Setting and main results}

\subsection{Entropy and pressure}

As in \cite{LPS95}, we want that the entropy of any sequence of random variables be well-defined. We add an extra hypothesis so that the same holds for the pressure. Let $(\mathcal{X}, \tau)$ be a (not necessarily Hausdorff) topological vector space and let $\mathcal{F}$ be a (not necessarily Borel) $\sigma$-algebra over $\mathcal{X}$ such that:
\begin{axioms}{LM}
\item \label{lm1} every point $x$ in $\mathcal{X}$ admits a local basis $\mathcal{V}_x$ made up of measurable neighborhoods;
\item \label{lm2} every continuous linear functional over $\mathcal{X}$ is a measurable map.
\end{axioms}
We will say that $(\mathcal{X}, \tau, \mathcal{F})$ is a \emph{locally measurable vector space} (\emph{l.m.v.s.}). Throughout the text, $\mathcal{X}^*$ denotes the topological dual of $\mathcal{X}$, i.e.\ the vector space of continuous linear functional over $\mathcal{X}$. Let $(Z_n)_{n \geqslant 1}$ be a sequence of $(\mathcal{X}, \mathcal{F})$-valued random variables and $(v_n)_{n \geqslant 1}$ a sequence of positive numbers diverging to $+\infty$. Let us define three important functions in the theory of large deviations (further details can be found in \cite{BaZ79}, \cite{LPS95}, and \cite{LeP95}). The \emph{entropy} of the sequence $(Z_n, v_n)_{n \geqslant 1}$ is the function $s : \mathcal{X} \rightarrow [-\infty , 0]$ defined by
\[
\forall x \in \mathcal{X} \quad s(x) := \inf_{V \in \mathcal{V}_x} \liminf_{n \to \infty} \frac{1}{v_n} \log \Prob(Z_n \in V),
\]
the \emph{entropy superior} of the sequence $(Z_n, v_n)_{n \geqslant 1}$ is the function $\overline{s} : \mathcal{X} \rightarrow [-\infty , 0]$ defined by
\[
\forall x \in \mathcal{X} \quad \overline{s}(x) := \inf_{V \in \mathcal{V}_x} \limsup_{n \to \infty} \frac{1}{v_n} \log \Prob(Z_n \in V),
\]
and the \emph{pressure} of the sequence $(Z_n, v_n)_{n \geqslant 1}$ is the function $p : \mathcal{X}^* \rightarrow [-\infty , +\infty]$ defined by
\[
\forall \lambda \in \mathcal{X}^* \quad p(\lambda) := \limsup_{n \to \infty} \frac{1}{v_n} \log \mathbb{E}\big( e^{v_n \lambda(Z_n)} \big).
\]
Note that the functions $s$ and $\overline{s}$ do not depend on the local bases $\mathcal{V}_x$ chosen in each point $x \in \mathcal{X}$ (hypothesis \ref{lm1}) and that $p$ is well defined (hypothesis \ref{lm2}). Remember that the \emph{Fenchel-Legendre transform} of a function $h : \mathcal{X}^* \to [-\infty , +\infty]$ is the function $h^* : \mathcal{X} \to [-\infty , +\infty]$
\[
\forall x \in \mathcal{X} \quad h^*(x) = \sup_{\lambda \in \mathcal{X}^*} \big( \lambda(x) - h(\lambda) \big) .
\]
(see \cite{Mor67} for a detailed study of the notion). It is always true that
\begin{equation}\label{ineq_ssp}
s \leqslant \overline{s} \leqslant -p^*,
\end{equation}
the second inequality being a consequence of Chebyshev's inequality. It appears that equalities in \eqref{ineq_ssp} are of greatest interest in large deviation theory. Indeed, by definition, the entropy $s$ is the greatest function that satisfies the lower bound:
\begin{axiom}{LB}
\item \label{lb} for all subset $G \in \mathcal{F}$, $\displaystyle{\liminf_{n \to \infty} \frac{1}{v_n} \log \Prob\big( Z_n \in G \big) \geqslant \sup_{x \in \groi} s(x)}$.
\end{axiom}
And the entropy superior $\overline{s}$ satisfies the compact upper bound:
\begin{axiom}{UB$_\textsf{k}$}
\item \label{ubk} for all relatively compact subset $K \in \mathcal{F}$, $\displaystyle{\limsup_{n \to \infty} \frac{1}{v_n} \log \Prob\big( Z_n \in K \big) \leqslant \sup_{x \in \overline{K}} \overline{s}(x)}$.
\end{axiom}
Recall that a set is compact if each of its open covers has a finite subcover (no separation axiom is required). And we say that a subset is relatively compact if its closure is compact. The proof of \ref{ubk} can be adapted from the one of \cite[Lemma 2.5]{LeP95} as follows: if $\sup\ensavec{\overline{s}(x)}{x \in \overline{K}} < M$, then, forall $x \in \overline{K}$, there exists $V_x \in \mathcal{V}_x$ such that $\limsup v_n^{-1} \log \Prob(x \in V_x) < M$; extracting a finite subcover of $\overline{K}$ by the $V_x$ and using the principle of the largest term (\cite[Lemma 2.3]{LeP95}), one gets the bound. Therefore, a natural sufficient condition so that $(Z_n, v_n)_{n \geqslant 1}$ satisfies a weak large deviation principle (i.e.\ \ref{lb} + \ref{ubk}) is that $s = \overline{s}$. Furthermore, the function $p^*$ is often easier to compute than $s$, which makes relevant the study of the equality $s = -p^*$.

The purpose of the text is to show that both l.m.v.s.\ and equalities in (\ref{ineq_ssp}) well behave with respect to inverse images and projective limits, in the spirit of Dawson-Gärtner theorem (see\emph{, e.g.}, \cite[Theorem 4.6.1]{DeZ93s}). Then we apply the result to the case of empirical means of independent and identically distributed random variables (Cramér's theory) and obtain a general weak large deviation principle with entropy $s=-p^*$.

\subsection{Inverse image of a locally measurable topological vector space}

Let $\mathcal{X}$ be a vector space, let $(\mathcal{X}_1, \tau_1, \mathcal{F}_1)$ be a l.m.v.s., and let $f_1 \colon \mathcal{X} \to \mathcal{X}_1$ be a linear map. Let $\tau$ (resp. $\mathcal{F}$) be the initial topology (resp. the initial $\sigma$-algebra) on $\mathcal{X}$ with repect to the linear map $f_1$, also called \emph{inverse image topology under $f_1$} (resp. \emph{$\sigma$-algebra generated by $f_1$}). It appears to be that $(\mathcal{X}, \tau, \mathcal{F})$ is a l.m.v.s.\ (see Propositions \ref{i_vois} and \ref{i_dual} below). We will say that $(\mathcal{X}, \tau, \mathcal{F})$ is the \emph{inverse image of the l.m.v.s.\ $(\mathcal{X}_1, \tau_1, \mathcal{F}_1)$ under $f_1$}.

\begin{theorem}[Linear inverse contraction principle] \label{ii}
Let $\mathcal{X}$ be a vector space, let $(\mathcal{X}_1, \tau_1, \mathcal{F}_1)$ be a l.m.v.s., and let $f_1 \colon \mathcal{X} \to \mathcal{X}_1$ be a linear map. Let $(\mathcal{X}, \tau, \mathcal{F})$ be the inverse image of the l.m.v.s.\ $(\mathcal{X}_1, \tau_1, \mathcal{F}_1)$ under $f_1$. Let $(Z_n)_{n \geqslant 1}$ be a sequence of $(\mathcal{X}, \mathcal{F})$-valued random variables and $(v_n)_{n \geqslant 1}$ a sequence of positive numbers diverging to $+\infty$. Define $s$, $\overline{s}$, and $p$ (resp.\ $s_1$, $\overline{s}_1$, and $p_1$) the entropy, the entropy superior, and the pressure of the sequence $(Z_n, v_n)_{n \geqslant 1}$ (resp.\ of the sequence $(f_1(Z_n), v_n)_{n \geqslant 1}$). Then,
\begin{equation}\label{ii_ssp}
s = s_1 \circ f_1, \quad \overline{s} = \overline{s}_1 \circ f_1, \quad\text{and} \quad p^* = p_1^* \circ f_1.
\end{equation}
\begin{enumerate}
\item \label{ii_ent} If $s_1 = \overline{s}_1$, then $s = \overline{s}$; in particular, $(Z_n, v_n)_{n \geqslant 1}$ satisfies a weak large deviation principle.
\item \label{ii_equiv} If $\overline{s}_1 = -p_1^*$, then $\overline{s} = -p^*$.
\item \label{ii_pgd} If $(f_1(Z_n), v_n)_{n \geqslant 1}$ satisfies a weak large deviation principle and each compact subset of $\mathcal{X}_1$ is measurable, then $(Z_n, v_n)_{n \geqslant 1}$ satisfies a weak large deviation principle.
\end{enumerate}
\end{theorem}

\subsection{Projective limits of locally measurable topological vector spaces}

Let $(I, \leqslant)$ be a directed preordered set. For all $i \in I$, let $(\mathcal{X}_i, \tau_i, \mathcal{F}_i)$ be a l.m.v.s.\ and, for all $(i, j) \in I^2$ with $i \leqslant j$, let $f_{i,j}$ be a map from $\mathcal{X}_j$ to $\mathcal{X}_i$. Suppose that
\begin{axioms}{PL}
\item \label{pl1} for all $i \in I$, $f_{i,i} = \id_{\mathcal{X}_i}$;
\item \label{pl2} for all $(i,j,k) \in I^3$ with $i \leqslant j \leqslant k$, $f_{i,k} = f_{i,j} \circ f_{j,k}$;
\item \label{pl3} for $(i, j) \in I^2$ with $i \leqslant j$, $f_{i,j}$ is a measurable, continuous, and linear map.
\end{axioms}
In particular, $(\mathcal{X}_i, f_{i,j})_{i \leqslant j}$ is a projective system of sets (see\emph{, e.g.}, \cite[III.52]{BouE}). Let $\mathcal{X}$ be its projective limit and, for all $i \in I$, let $f_i$ be the canonical projection from $\mathcal{X}$ to $\mathcal{X}_i$. Let $\tau$ (resp. $\mathcal{F}$) be the initial topology (resp. the initial $\sigma$-algebra) on $\mathcal{X}$ with respect to the family of maps $(f_i)_{i \in I}$. It appears to be that $(\mathcal{X}, \tau, \mathcal{F})$ is also a l.m.v.s.\ (see Propositions \ref{i_vois} and \ref{i_dual} below). We will say that $(\mathcal{X}, \tau, \mathcal{F})$ is the \emph{projective limit of the projective system of l.m.v.s.}\ $(\mathcal{X}_i, \tau_i, \mathcal{F}_i, f_{i,j})_{i \leqslant j}$.

\begin{theorem}[Linear Dawson-Gärtner]\label{ldg}
Let $(\mathcal{X}_i, \tau_i, \mathcal{F}_i, f_{i,j})_{i \leqslant j}$ be a projective system of l.m.v.s.\ and $(\mathcal{X}, \tau, \mathcal{F})$ its projective limit. For all $i \in I$, let $f_i$ be the canonical projection from $\mathcal{X}$ to $\mathcal{X}_i$. Let $(Z_n)_{n \geqslant 1}$ be a sequence of $(\mathcal{X}, \mathcal{F})$-valued random variables and $(v_n)_{n \geqslant 1}$ a sequence of positive numbers diverging to $+\infty$. Define $s$, $\overline{s}$, and $p$ (resp. $s_i$, $\overline{s}_i$, and $p_i$) the entropy, the entropy superior, and the pressure of the sequence $(Z_n, v_n)_{n \geqslant 1}$ (resp. of the sequence $(f_i(Z_n), v_n)_{n \geqslant 1}$). Then,
\begin{equation}\label{ldg_ssp}
s = \inf_{i \in I} s_i \circ f_i, \quad \overline{s} = \inf_{i \in I} \overline{s}_i \circ f_i, \quad\text{and} \quad p^* = \sup_{i \in I} p_i^* \circ f_i.
\end{equation}
\begin{enumerate}
\item \label{ldg_ent} If, for all $i \in I$, $s_i = \overline{s}_i$, then $s = \overline{s}$; in particular, $(Z_n, v_n)_{n \geqslant 1}$ satisfies a weak large deviation principle.
\item \label{ldg_equiv} If, for all $i \in I$, $\overline{s}_i = -p_i^*$, then $\overline{s} = -p^*$.
\item \label{ldg_pgd} If, for all $i \in I$, $(f_i(Z_n), v_n)_{n \geqslant 1}$ satisfies a weak large deviation principle and each compact subset of $\mathcal{X}_i$ is measurable, then $(Z_n, v_n)_{n \geqslant 1}$ satisfies a weak large deviation principle.
\end{enumerate}
\end{theorem}

\section{Proofs of Theorems \ref{ii} and \ref{ldg}}

Let us kill two birds with one stone. Let $(\mathcal{X}, \tau, \mathcal{F})$ be either an inverse image structure or a projective limit structure. The structure $(\mathcal{X}, \tau, \mathcal{F})$ is defined as an initial structure with respect to a family of linear maps $f_i \colon \mathcal{X} \to \mathcal{X}_i$ ($i \in I$) such that $(I, \leqslant)$ is a directed preordered set ($I = \{ 1 \}$ in the case of the inverse image). Moreover, we have the following key result, where, for all $i \in I$ and all point $x_i \in \mathcal{X}_i$, $\mathcal{V}_{i, x_i}$ denotes a local basis of $x_i$ made up of measurable neighborhoods.
\begin{proposition}\label{i_vois}
For all $x \in \mathcal{X}$, the set
\[
\mathcal{V}_x := \{ f_i^{-1}(V_i) \, ; \, i \in I, \, V_i \in \mathcal{V}_{i, f_i(x)} \}
\]
is a local basis of neighborhoods of $x$. In particular, $(\mathcal{X}, \tau, \mathcal{F})$ satisfies \ref{lm1}.
\end{proposition}

\begin{proof}
In the case of an inverse image structure, it stems from the definition of the inverse image topology. And, as for a projective limit structure, it is proved in \cite[I.29]{Bou71}.
\end{proof}

\begin{proposition}\label{i_dual}
The topological dual of $\mathcal{X}$ is
\[
\mathcal{X}^* = \{ \lambda_i \circ f_i \, ; \, i \in I, \, \lambda_i \in \mathcal{X}_i^* \}.
\]
In particular, $(\mathcal{X}, \tau, \mathcal{F})$ satisfies \ref{lm2}.
\end{proposition}

\begin{proof}
Let $\lambda \in \mathcal{X}^*$. Since $\lambda$ is continuous, there is a set $V \in \mathcal{V}_0$ such that $V \subset \{ \lambda < 1 \}$. According to \Cref{i_vois}, there exists $i \in I$ and $V_i \in \mathcal{V}_{i,0}$ such that $V = f_i^{-1}(V_i)$. Then, for all $x \in \mathcal{X}$, the equality $f_i(x) = 0$ entails $\lambda(x) = 0$: indeed, if $\lambda(x) \neq 0$, then there is a real number $t$ such that $\lambda(tx) > 1$, so $f_i(x) \notin V_i$, whence $f_i(x) \neq 0$. Therefore, there exists a linear functional $\tilde{\lambda}$ such that the diagram
%
%
%
%
\begin{center}
\begin{tikzcd}
 & \R \\
\mathcal{X} \arrow[ur, "\lambda"] \arrow[r, "q"] \arrow[dr, "f_i"] & \mathcal{X}/\Ker(f_i) \arrow[u, "\tilde{\lambda}"] \arrow[d, "\tilde{f}_i"] \\
 & f_i(\mathcal{X}) \arrow[r, hook] & \mathcal{X}_i
\end{tikzcd}
\end{center}
is commutative, where $q$ is the quotient map and $\tilde{f}_i$ the isomorphism induced by $f_i$. The map $\tilde{\lambda}_i = \tilde{\lambda} \circ \tilde{f}_i^{-1}$ is a linear functional on $f_i(\mathcal{X})$ and $V_i \cap f_i(\mathcal{X}) \subset \{ \tilde{\lambda}_i < 1 \}$: indeed, for $x_i \in V_i \cap f_i(\mathcal{X})$, there exists $x \in V = f_i^{-1}(V_i)$ such that $f(x) = x_i$, whence
\[
\tilde{\lambda}_i(x_i) = \tilde{\lambda} \circ \tilde{f}_i^{-1} \circ f_i(x) = \tilde{\lambda} \circ \tilde{f}_i^{-1} \circ \tilde{f}_i \circ q(x) = \tilde{\lambda} \circ q(x) = \lambda(x) < 1.
\]
So $\tilde{\lambda}_i$ is a continuous linear functional on $f_i(\mathcal{X})$. According to the Hahn-Banach theorem, there exists a linear extension $\lambda_i \in \mathcal{X}_i^*$ of $\tilde{\lambda}_i$ and it satisfies $\lambda = \lambda_i \circ f_i$.
\end{proof}

\begin{proof}[Proof of Theorems \ref{ii} and \ref{ldg}]
Let us begin with the proof of
\begin{equation}\label{i_ssp}
s = \inf_{i \in I} s_i \circ f_i, \quad \overline{s} = \inf_{i \in I} \overline{s}_i \circ f_i, \quad\text{and} \quad p^* = \sup_{i \in I} p_i^* \circ f_i.
\end{equation}
Using \Cref{i_vois}, for all $x \in \mathcal{X}$,
\begin{align*}
s(x) &= \inf_{V \in \mathcal{V}_x} \liminf_{n \to \infty} \frac{1}{v_n} \log \Prob(Z_n \in V)\\
 &= \inf_{i \in I} \inf_{V_i \in \mathcal{V}_{i, f_i(x)}} \liminf_{n \to \infty} \frac{1}{v_n} \log \Prob \big( f_i(Z_n) \in V_i \big) = \inf_{i \in I} s_i \big( f_i(x) \big).
\end{align*}
The proof for $\overline{s}$ is similar. Finally, using \Cref{i_dual},
\[
p^*(x) = \sup_{\lambda \in \mathcal{X}^*} \big( \lambda(x) - p(\lambda) \big) = \sup_{i \in I} \sup_{\lambda_i \in \mathcal{X}_i^*} \big( \lambda_i \circ f_i(x) - p(\lambda_i \circ f_i) \big) = \sup_{i \in I} p_i^*\big( f_i(x) \big).
\]
Assertions \ref{ldg_ent} and \ref{ldg_equiv} follow immediately. Let us prove assertion \ref{ldg_pgd}. Let $K \in \mathcal{F}$ be a relatively compact subset of $\mathcal{X}$ and $i \in I$. Then, using the fact that $f_i(\overline{K})$ is compact, thus measurable,
\[
\limsup_{n \to \infty} \frac{1}{v_n} \log \Prob(Z_n \in K) \leqslant \limsup_{n \to \infty} \frac{1}{v_n} \log \Prob\big( f_i(Z_n) \in f_i(\overline{K}) \big) \leqslant \sup_{ f_i( \overline{K} )  } s_i = \sup_{\overline{K}} s_i \circ f_i
\]
and it remains to prove the min-max inequality
\begin{equation} \label{minmax}
\inf_{i \in I} \sup_{\overline{K}} s_i \circ f_i \leqslant \sup_{\overline{K}} \inf_{i \in I} s_i \circ f_i = \sup_{\overline{K}} s.
\end{equation}
Let $\alpha > 0$. Using \eqref{i_ssp}, for all $x \in \overline{K}$, there exists $i(x) \in I$ such that
\[
s_{i(x)} \circ f_{i(x)} (x) \leqslant \max\big( s(x) + \alpha/2, -1/\alpha - \alpha/2 \big).
\]
Since $s_{i(x)} \circ f_{i(x)}$ is upper semi-continuous, there exists $V(x) \in \mathcal{V}_x$ such that
\[
\sup_{V(x)} s_{i(x)} \circ f_{i(x)} \leqslant \max\big( s_{i(x)} \circ f_{i(x)} (x) + \alpha/2, -1/\alpha \big).
\]
Hence
\[
\sup_{V(x)} s_{i(x)} \circ f_{i(x)} \leqslant \max\big( s(x) + \alpha, -1/\alpha \big).
\]
The cover $\ensavec{V(x)}{x \in \overline{K}}$ of $\overline{K}$ admits a finite subcover $\ensavec{V(x_k)}{k \in \{ 1, \ldots , r \}}$. Since $I$ is a directed set, there exists an upper bound $j \in I$ of $\{ i(x_1), \ldots , i(x_r) \}$. Let us check that
\begin{equation}\label{ineg_s}
\forall k \in \{ 1, \ldots , r \} \qquad s_j \circ f_j \leqslant s_{i(x_k)} \circ f_{i(x_k)}.
\end{equation}
For $i \leqslant j$,
\[
s_i \circ f_i = s_i \circ f_{i,j} \circ f_j,
\]
so we only need to prove that $s_j \leqslant s_i \circ f_{i,j}$. For all $x \in X_j$, since $f_{ij}^{-1}(V_i) \in \mathcal{V}_{j, x}$,
\begin{align*}
\inf_{V_j \in \mathcal{V}_{j, x}} \liminf_{n \to \infty} \frac{1}{v_n} \log \Prob\big( f_j(Z_n) \in V_j \big)
 & \leqslant \inf_{V_i \in \mathcal{V}_{i, f_{i,j}(x)}} \liminf_{n \to \infty} \frac{1}{v_n} \log \Prob\big( f_{ij} \circ f_j(Z_n) \in V_i \big)\\
 & = \inf_{V_i \in \mathcal{V}_{i, f_{i,j}(x)}} \liminf_{n \to \infty} \frac{1}{v_n} \log \Prob\big( f_i(Z_n) \in V_i \big)
\end{align*}
and equation \eqref{ineg_s} is proved. Hence
\[
\sup_{\overline{K}} s_j \circ f_j \leqslant \max\bigg( \max_{1 \leqslant k \leqslant r} s(x_k) + \alpha, -1/\alpha \bigg) \leqslant \max\bigg( \sup_{\overline{K}} s + \alpha, -1/\alpha \bigg).
\]
Conclude to \eqref{minmax} by letting $\alpha \to 0$.
\end{proof}

\section{Application to Cramér's theory} \label{sec:applic_cramer}

Let $(\mathcal{X}, \tau, \mathcal{F})$ be a l.m.v.s. Let $(X_n)_{n \geqslant 1}$ be a sequence of independent and identically distributed (i.i.d.)\ random variables over $(\mathcal{X}, \mathcal{F})$. In this section, we consider the case when $Z_n$ is the empirical mean $\overline{X}_n = (X_1 + X_2 + \cdots + X_n)/n$. In order that $Z_n$ be well defined, we introduce extra assumptions:
\begin{axioms}{LM}
\setcounter{axioms}{2}
\item \label{lm3} the addition $(x,y) \in (\mathcal{X}^2, \mathcal{F}^{\otimes 2}) \mapsto x+y \in (\mathcal{X}, \mathcal{F})$ is measurable;
\item \label{lm4} the $\sigma$-algebra $\mathcal{F}$ is dilatation-invariant.
\end{axioms}
As before, we define the entropy $s$, the entropy superior $\overline{s}$, and the pressure $p$ associated to the sequence $(\overline{X}_n, n)_{n \geqslant 1}$, or shortly to the distribution $\mu$ of $X_1$. In this context, the pressure simplifies like this:
\[
p(\lambda) = \log \Espe[e^{\lambda(X_1)}] .
\]
\begin{remark}
All the spaces we will consider will be locally convex spaces in which each point admits a local basis made up of open convex measurable neighborhoods (see \cite[Chapter 4]{Cer07} who introduces this natural assumption). In this context, we also have $s = \overline{s}$: the proof of \cite[Proposition 2]{Petit_2018_CramersTheoremBanach} adapts without change.
\end{remark}

It can be found in standard texts that any separable Banach space $(\mathcal{B}, \tau_{\norme{\cdot}}, \sigma(\tau_{\norme{\cdot}}))$ endowed with its Borel $\sigma$-algebra satisfies \ref{lm1}, \ref{lm2}, \ref{lm3}, and \ref{lm4}. Moreover, in this setting:
\begin{theorem} \label{cramerbanachsep}
Let $(X_n)_{n \geqslant 1}$ be a sequence of i.i.d.\ random variables on a separable Banach space $(\mathcal{B}, \tau_{\norme{\cdot}}, \sigma(\tau_{\norme{\cdot}}))$. The sequence $(\overline{X}_n, n)_{n \geqslant 1}$ satisfies a weak large deviation principle with entropy $s = -p^*$ where $p(\lambda) = \log \Espe\big(e^{\lambda(X_1)}\big)$.
\end{theorem}

\begin{proof}
See, \emph{e.g.}, \cite[Theorem 6.1.3]{DeZ93s}, \cite[Corollary 16.4]{Cer07}, or \cite[Theorems 12.1 and 12.4]{Petit_2018_CramersTheoremBanach} for different proofs.
\end{proof}

In this section we provide general explicit locally convex spaces which satisfy \ref{lm1}, \ref{lm2}, \ref{lm3}, and \ref{lm4}, in which the equality $s = -p^*$ is always true. In \Cref{sec:counterexample1}, we provide an example where the space satisfy \ref{lm1}, \ref{lm2}, \ref{lm3}, and \ref{lm4}, but in which $s = \overline{s} \neq -p^*$.

\subsection{Cramér's theorem in separable normed vector spaces}

Here, we relax the hypothesis of completeness in \Cref{cramerbanachsep}. Let $(\mathcal{N}, \tau_{\norme{\cdot}}, \sigma(\tau_{\norme{\cdot}}))$ be a separable normed vector space (not necessarily complete) endowed with its Borel $\sigma$-algebra. Obviously, $(\mathcal{N}, \tau_{\norme{\cdot}}, \sigma(\tau_{\norme{\cdot}}))$ satisfies \ref{lm1}, \ref{lm2} and \ref{lm4}. And, since $\mathcal{N}$ is second countable, \ref{lm3} is satisfied.

\begin{theorem} \label{cramerevns}
Let $(X_n)_{n \geqslant 1}$ be a sequence of i.i.d.\ random variables on a separable normed vector space $(\mathcal{N}, \tau_{\norme{\cdot}}, \sigma(\tau_{\norme{\cdot}}))$. The sequence $(\overline{X}_n, n)_{n \geqslant 1}$ satisfies a weak large deviation principle with entropy $s = -p^*$ where $p(\lambda) = \log \Espe\big(e^{\lambda(X_1)}\big)$.
\end{theorem}

\begin{proof}
Just consider the completion $(\hat{\mathcal{N}}, \hat{\tau})$ of $(\mathcal{N}, \tau_{\norme{\cdot}})$, which is a separable Banach space. Let $i \colon \mathcal{N} \hookrightarrow \hat{\mathcal{N}}$ be the canonical inclusion. By construction, $\tau_{\norme{\cdot}}$ is the inverse image topology of $\hat{\tau}$ with respect to $i$. And the Borel $\sigma$-algebra satisfies
\[
\sigma(\tau_{\norme{\cdot}}) = \sigma\big( i^{-1}(\hat{\tau}) \big) = i^{-1}(\sigma(\hat{\tau})) .
\]
So $(\mathcal{N}, \tau_{\norme{\cdot}}, \sigma(\tau_{\norme{\cdot}}))$ is the inverse image of the l.m.v.s.\ $(\hat{\mathcal{N}}, \hat{\tau}, \sigma(\hat{\tau}))$ with respect to $i$. According to \Cref{cramerbanachsep}, $(i(\overline{X}_n), n)_{n \geqslant 1}$ satisfies a weak large deviation principle with entropy $\hat{s} = -\hat{p}^*$ where $\hat{p}(\lambda) = \log \Espe(e^{\lambda(i(X_1))})$. Eventually apply \Cref{ii} to conclude with $s = \hat{s}|_{\mathcal{N}}$ and $p^* = \hat{p}^*|_{\mathcal{N}}$.
\end{proof}

\subsection{Cramér's theorem in locally convex separable spaces} \label{sec:cramerlcss}


The following simple construction makes explicit the settings developped in \cite{BaZ79}, \cite{DeZ93s}, and \cite{Cer07}, and generalizes both Cramér's theorem in separable Banach spaces and Sanov's theorem. It will be generalized even more in the next section, at the expense of technical notions of set theory. Let $\mathcal{X}$ be a vector space and consider a family $\mathcal{C}$ of subsets of $\mathcal{X}$ satisfying
\begin{axiom}{LCS}
\item \label{lcs} every $C \in \mathcal{C}$ is a symmetric algebraically open convex set containing $0$ such that the locally convex topology generated by $C$ is separable (or equivalently second-countable).
\end{axiom}

For all $x \in \mathcal{X}$, define
\[
\mathcal{C}_x = \ensavec{x + \bigcap_{i=1}^r \lambda_i C_i}{r \geqslant 0, \, C_i \in \mathcal{C}, \, \lambda_i > 0}.
\]
Then, $\mathcal{C}_0$ is a local basis of $0$ for a locally convex topology $\tau(\mathcal{C})$ on $\mathcal{X}$ (see\emph{, e.g.}, \cite[II.25]{BouEVT}). As for the $\sigma$-algebra, let
\[
\mathcal{F}(\mathcal{C}) = \sigma\bigg( \bigcup_{x \in \mathcal{X}} \mathcal{C}_x \bigg).
\]

Let us make several points on the setting. First, if $C$ is an algebraically open convex subset containing $0$, then $C \cap (-C)$ is a symmetric algebraically open convex subset containing $0$, so the symmetry can be assumed without any restriction, but it is handy when introducing seminorms. Secondly, $\mathcal{F}(\mathcal{C})$ is not necessarily the Borel $\sigma$-algebra $\sigma(\tau(\mathcal{C}))$: it is the case when $\tau(\mathcal{C})$ is second-countable. Thirdly, each $C \in \mathcal{C}_0$ is symmetric, and the locally convex topology generated by $C$ is separable: indeed, the upper bound of finitely many second-countable topologies is still second-countable. However, $\tau(\mathcal{C})$ is not necessarily separable (a fortiori not second-countable). For instance, take $\mathcal{X} = \R^{\R}$ endowed with the product topology $\tau$ and the cylinder $\sigma$-algebra $\mathcal{F}$. The topology $\tau$ is not separable and the space $(\mathcal{X}, \tau, \mathcal{F})$ is obtained by the previous construction with $\mathcal{C}$ being the set of all open half-spaces $C = \enstq{f \in \mathcal{X}}{f(x) < 1}$ ($x \in \R$).

In the following, we will denote by $M_C$ the Minkowski functional of any convex subset $C$ of $\mathcal{X}$ containing $0$ (see\emph{, e.g.}, \cite{Zal02}). In particular, if $C \in \mathcal{C}$, $M_C$ is a seminorm.

\begin{proposition} \label{prop:lcss_lmtvs}
$(\mathcal{X}, \tau(\mathcal{C}), \mathcal{F}(\mathcal{C}))$ satisfies \ref{lm1}, \ref{lm2}, \ref{lm3}, and \ref{lm4}.
\end{proposition}

As a result, we will say that $(\mathcal{X}, \tau(\mathcal{C}), \mathcal{F}(\mathcal{C}))$ is a \emph{locally convex separable space} (\emph{l.c.s.s.).}

\begin{proof}
By construction, \ref{lm1} and \ref{lm4} are satistified. To prove \ref{lm2}, let $\lambda \in \mathcal{X}^*$ and let $H$ be the open half-space $\enstq{x \in \mathcal{X}}{\lambda(x) < 1}$. We only need to prove that $H$ is measurable. Since $H$ is a neighbourhood of $0$, there exists $C \in \mathcal{C}_0$ such that $C \subset H$. Let us show that $H$ is a countable union of homothetic transforms of $C$, so that $H$ is measurable. Let $Q$ be a countable subset of $\mathcal{X}$ which is dense with respect to the locally convex topology $\tau_C$ generated by $C$. We have
\begin{equation} \label{handc}
H \supseteq \bigcup_{\substack{ u \in Q , \, r \in \mathbb{Q}_+ \\ u + rC \subset H} } u + rC
\end{equation}
and we want to show the converse inclusion. Let $x \in H$. Since $H$ is open, $M_H(x) < 1$ and there exists $r \in \mathbb{Q}_+$ such that $0 < r < 1 - M_H(x)$; so $x + rC \subset H$. The set $x + rC/2$ is open with respect to $\tau_C$, so contains a point $u$ of $Q$. Then, $x \in u + rC/2 \subset x + rC \subset H$ and the converse inclusion of \eqref{handc} is proved.

Now we prove \ref{lm3}, i.e.\ the addition $(x, y) \in (\mathcal{X}^2, \mathcal{F}(\mathcal{C})^{\otimes 2}) \mapsto x + y \in (\mathcal{X}, \mathcal{F}(\mathcal{C}))$ is measurable. Let $C \in \mathcal{C}_0$. We only need to prove that $\enstq{(x, y) \in \mathcal{X}^2}{x + y \in C} \in \mathcal{F}(\mathcal{C})^{\otimes 2}$. Remember that the locally convex topology $\tau_C$ generated by $C$ is separable. Let $Q$ be a countable subset of $\mathcal{X}$ which is dense with respect to $\tau_C$. Let us show that
\[
\enstq{(x, y) \in \mathcal{X}^2}{x + y \in C} = \bigcup_{\substack{u \in Q \\ r \in \Q}} \enstq{(x, y) \in \mathcal{X}^2}{M_C(x - u) < r, \, M_C(y + u) < 1 - r}.
\]
To show $\subset$, let $(x, y) \in \mathcal{X}^2$ such that $x + y \in C$. Let $a \defeq M_C(x+y) < 1$. Let $r \in \Q \cap \intervalleoo{0}{(1-a)/2}$. The nonempty open subset $x+rC$ contains a point $u$ of $Q$. Then $M_C(x-u) < r$ (since $C$ is symmetric) and
\[
M_C(y+u) \leqslant M_C(y+x) + M_C(u-x) < a + r < 1 - r.
\]
To show $\supset$, just notice that, for each $(x, y, u) \in \mathcal{X}^3$,
\[
M_C(x+y) \leqslant M_C(x-u) + M_C(y+u).
\]
\end{proof}

\begin{theorem} \label{cramerlcs}
Let $(X_n)_{n \geqslant 1}$ be a sequence of i.i.d.\ random variables on a l.c.s.s.\ $(\mathcal{X}, \tau(\mathcal{C}), \mathcal{F}(\mathcal{C}))$. The sequence $(\overline{X}_n, n)_{n \geqslant 1}$ satisfies a weak large deviation principle with entropy $s = -p^*$ where $p(\lambda) = \log \Espe\big(e^{\lambda(X_1)}\big)$.
\end{theorem}

\begin{proof}
It appears to be that the previous construction of $(\mathcal{X}, \tau(\mathcal{C}), \mathcal{F}(\mathcal{C}))$ is equivalent to the following one. For each $C \in \mathcal{C}_0$, let $\tau_C$ (resp.\ $\mathcal{F}_C$) be the locally convex topology generated by $C$ (resp.\ the $\sigma$-algebra generated by all homothetic transforms of $C$). Then, $(\mathcal{X}, \tau(\mathcal{C}), \mathcal{F}(\mathcal{C}))$ is the projective limit of the projective system of l.m.v.s.\ $(\mathcal{X}, \tau_C, \mathcal{F}_C, \id_{\mathcal{X}})_{C \supseteq B}$. Now, for $C \in \mathcal{C}_0$, consider the Hausdorff space $(\mathcal{N}_C, \tilde{\tau}_C)$ canonically associated to $(\mathcal{X}, \tau_C)$ and $f_C \colon \mathcal{X} \to N_C$ the canonical projection. Then, since $\tau_C$ is separable, $\mathcal{N}_C$ is a separable normed vector space, and $(\mathcal{X}, \tau_C, \mathcal{F}_C)$ is the inverse image of the l.m.v.s.\ $(\mathcal{N}_C, \tilde{\tau}_C, \sigma(\tilde{\tau}_C))$ under $f_C$ (notice that, since $(\mathcal{N}_C, \tilde{\tau}_C$) is a separable normed space, the $\sigma$-algebra generated by the open balls is equal to the Borel $\sigma$-algebra $\sigma(\tilde{\tau}_C)$). By \Cref{cramerevns}, the conclusion holds in $(\mathcal{N}_C, \tilde{\tau}_C, \sigma(\tilde{\tau}_C))$. And the conclusion holds in $(\mathcal{X}, \tau(\mathcal{C}), \mathcal{F}(\mathcal{C}))$ through theorems \ref{ii} and \ref{ldg}.
\end{proof}

The model of \Cref{cramerlcs} is closely related to the models of \cite{BaZ79} and \cite{Cer07}. In \cite{BaZ79}, they consider a Hausdorff locally convex topological vector space $(\mathcal{X}, \tau)$ endowed with the Borel $\sigma$-algebra $\sigma(\tau)$ and assume conditions of measurability so that $\overline{X}_n$ be measurable. Remark that $p$ is automatically defined when the $\sigma$-algebra is the Borel $\sigma$-algebra. Here, we do not suppose the $\sigma$-algebra to be the Borel $\sigma$-algebra. In \cite{Cer07}, the $\sigma$-algebra is not supposed to be the Borel $\sigma$-algebra either, but it is assumed that the vector space operations are measurable, that each point of $\mathcal{X}$ admits a local basis of open convex measurable subsets of $\mathcal{X}$, and that the continuous linear functionals are measurable. Here, we do not require the scalar multiplication to be measurable. More importantly, \Cref{cramerlcs} provides a simple constructive model for Cramér's theory: separability appears to be the crux of the matter of measurability (even in the generalization exposed in \cref{sec:cramerlcms}). And the model embraces standard applications (see\emph{, e.g.}, \cite[chap.\ 16]{Cer07}). Indeed, we obtain the weak large deviation principle and the equality $s = -p^*$ when:
\begin{itemize}
\item $\mathcal{X}$ is a separable normed vector space endowed with the Borel $\sigma$-algebra ($\mathcal{C} = \{ B(0,1) \}$), in particular when $\mathcal{X}$ is a separable Banach space (standard Cramér's theorem);
\item $\mathcal{X}$ is a vector space in duality with another one $\mathcal{Y}$, and $\mathcal{X}$ is endowed with the weak-topology and the cylinder $\sigma$-algebra ($\mathcal{C}$ is the collection of cylinders $\{ \abs{\langle \cdot , y \rangle} < 1 \}$ for $y \in \mathcal{Y}$.), which allows to recover Sanov's theorem in the $\tau$-topology.
\end{itemize}

\subsection{Cramér's theorem in locally convex measurable spaces} \label{sec:cramerlcms}

In the previous section, we have obtained Cramér's theorem in any projective limit of separable normed spaces. Here, we relax the hypothesis of separability. Let $(\mathcal{N}, \tau_{\norme{\cdot}})$ be a normed space. We say that $(\mathcal{N}, \tau_{\norme{\cdot}}, \sigma(\tau_{\norme{\cdot}}))$ is a \emph{measurable normed space} if the addition $(x,y) \mapsto x+y$ is measurable with respect to the $\sigma$-algebras $\sigma(\tau_{\norme{\cdot}}) \otimes \sigma(\tau_{\norme{\cdot}})$ and $\sigma(\tau_{\norme{\cdot}})$. In particular, $(\mathcal{N}, \tau_{\norme{\cdot}}, \sigma(\tau_{\norme{\cdot}}))$ satisfies \ref{lm1}, \ref{lm2}, \ref{lm3}, and \ref{lm4}.

\begin{theorem} \label{cramerevnm}
Let $(X_n)_{n \geqslant 1}$ be a sequence of i.i.d.\ random variables on a measurable normed space $(\mathcal{N}, \tau_{\norme{\cdot}}, \sigma(\tau_{\norme{\cdot}}))$. The sequence $(\overline{X}_n, n)_{n \geqslant 1}$ satisfies a weak large deviation principle with entropy $s = -p^*$ where $p(\lambda) = \log \Espe\big(e^{\lambda(X_1)}\big)$.
\end{theorem}

\begin{remark}
At this point, we can say that the equality $s = -p^*$ for empirical means of i.i.d.\ random variables is always true in the context of normed spaces with their Borel $\sigma$-algebra, as soon as the question makes sense. That is the reason why the example given in \Cref{sec:counterexample1} requires a little more effort.
\end{remark}

As in the construction in \Cref{sec:cramerlcss}, let us call \emph{locally convex measurable space} (\emph{l.c.m.s.}) any $(\mathcal{X}, \tau(\mathcal{C}), \mathcal{F}(\mathcal{C}))$, $\mathcal{C}$ being any family of subsets of $\mathcal{X}$ satisfying
\begin{axiom}{LCM}
\item \label{lcm} every $C \in \mathcal{C}$ is a symmetric algebraically open convex set containing $0$ such that $(\mathcal{N}_C, \tilde{\tau}_C)$ is a measurable normed space,
\end{axiom}
where we recall that $(\mathcal{N}_C, \tilde{\tau}_C)$ is the Hausdorff space canonically associated to $(\mathcal{X},\tau_C)$, $\tau_C$ being the locally convex topology generated by $C$. Following the proof of \Cref{cramerlcs}, we see that a l.c.m.s.\ is a projective limit of inverse images of measurable normed spaces and we deduce the following result.

\begin{theorem} \label{cramerlcm}
Let $(X_n)_{n \geqslant 1}$ be a sequence of i.i.d.\ random variables on a l.c.m.s.\ $(\mathcal{X}, \tau(\mathcal{C}), \mathcal{F}(\mathcal{C}))$. The sequence $(\overline{X}_n, n)_{n \geqslant 1}$ satisfies a weak large deviation principle with entropy $s = -p^*$ where $p(\lambda) = \log \Espe\big(e^{\lambda(X_1)}\big)$.
\end{theorem}

The proof of \Cref{cramerevnm} requires some technical notions of set theory. Essentially, we show that any probability measure over a measurable normed space is supported by a separable subspace and apply \Cref{cramerevns}. Let $\kappa$ be the \emph{density} of $\mathcal{N}$, i.e.\ the smallest cardinal $\lambda$ such that there exists a dense subset $D$ of $X$ of cardinality $\lambda$. We denote by $\enspart(E)$ the power set of the set $E$.

\begin{lemma} \label{lem:Tal79}
$\enspart(\kappa \times \kappa) = \enspart(\kappa) \otimes \enspart(\kappa)$.
\end{lemma}

\begin{proof}
The result is proved in \cite{Tal79}.
\end{proof}


\begin{lemma} \label{lem:Kun68}
There is no probability measure over $(\kappa, \enspart(\kappa))$ such that every singleton has measure zero.
\end{lemma}

\begin{proof}
\Cref{lem:Tal79} implies (see \cite[Lemma 12.2]{Kun68}) that $\kappa \leqslant 2^{\aleph_0}$ and, for all $\lambda \leqslant \kappa$, $\enspart(\lambda \times \lambda) = \enspart(\lambda) \otimes \enspart(\lambda)$. Then, applying \cite[Theorem 12.9]{Kun68}, we get that, for all $\lambda \leqslant \kappa$, $\lambda$ is not real-valued measurable (see \cite[Definition 11.1]{Kun68}. And the conclusion stems from \cite[Theorem 1.4]{Dra74}.
\end{proof}

Now, let $\mu$ be a probability measure over $(\mathcal{N}, \tau_{\norme{\cdot}}, \sigma(\tau_{\norme{\cdot}}))$.

\begin{lemma} \label{lem:Mar48}
There exists a closed separable subspace $\tilde{\mathcal{N}}$ of $\mathcal{N}$ such that $\mu(\tilde{\mathcal{N}}) = 1$.
\end{lemma}

\begin{proof}
Since there is no diffuse probability measure over $(\kappa, \enspart(\kappa))$ (see \Cref{lem:Kun68}), we can apply \cite[Theorem III]{Mar48} and deduce that $\mathcal{N} = S \cup Z$ where $S$ is a separable subset of $\mathcal{N}$, and $\mu(Z) = 0$ ($Z$ is the reunion of the negligeable open subsets of $\mathcal{N}$  and $S = \mathcal{N} \setminus Z$). Then, define $\tilde{\mathcal{N}} = \overline{\Vect(S)}$ the closed linear span of $S$. $\tilde{\mathcal{N}}$ is indeed separable since, if $D$ is a countable dense subset of $S$, then $\Vect_\Q(D)$ is a countable dense subset of $\tilde{\mathcal{N}}$.
\end{proof}

\begin{proof}[Proof of \Cref{cramerevnm}]
Let $\mu$ be the distribution of $X_1$. Let $\tilde{\mathcal{N}}$ be a closed separable subspace of $\mathcal{N}$ such that $\mu(\tilde{\mathcal{N}}) = 1$ (see \Cref{lem:Mar48}). Let $\tilde{s}$ (resp.\ $\tilde{p}$) be the entropy (resp.\ the pressure) associated to $\tilde{\mu} \defeq \mu|_{\tilde{\mathcal{N}}}$. \Cref{cramerevns} implies that $\tilde{s} = \tilde{p}$. Now,
\begin{equation} \label{eq:s}
s(x) = \begin{cases}
\tilde{s}(x) & \text{if $x \in \tilde{\mathcal{N}}$} \\
- \infty & \text{if $x \in \mathcal{N} \setminus \tilde{\mathcal{N}}$.}
\end{cases}
\end{equation}
Moreover, if $x \in \tilde{\mathcal{N}}$, since $\mu(\mathcal{N} \setminus \tilde{\mathcal{N}}) = 0$ and since the restriction $\tilde{\lambda}$ of any $\lambda \in \mathcal{N}^*$ to $\tilde{\mathcal{N}}$ is in $\tilde{\mathcal{N}}^*$,
\begin{align} \label{eq:p_Ntilde}
p^*(x)
  & = \sup_{\lambda \in \mathcal{N}^*} \biggl( \lambda(x) - \log \int_{\tilde{\mathcal{N}}} e^{\lambda(y)} d\mu(y) \biggr) \\
  & = \sup_{\tilde{\lambda} \in \tilde{\mathcal{N}}^*} \biggl( \tilde{\lambda}(x) - \log \int_{\tilde{\mathcal{N}}} e^{\tilde{\lambda}(y)} d\tilde{\mu}(y) \biggr) \nonumber \\
  & = \tilde{p}^*(x) . \nonumber 
\end{align}
And, if $x \in \mathcal{N} \setminus \tilde{\mathcal{N}}$, the Hahn-Banach theorem provides a $\lambda \in \mathcal{N}^*$ such that $\lambda(x) = 1$ and $\lambda$ is zero on $\tilde{\mathcal{N}}$; therefore, since $\mu(\mathcal{N} \setminus \tilde{\mathcal{N}}) = 0$,
\begin{equation} \label{eq:p_notinNtilde}
p^*(x) \geqslant \sup_{t \geqslant 0} \biggl( t\lambda(x) - \log \int_{\tilde{\mathcal{N}}} e^{t\lambda(y)} d\mu(y) \biggr) = \sup_{t \geqslant 0} t = + \infty .
\end{equation}
The conclusion follows from \crefrange{eq:s}{eq:p_notinNtilde}.
\end{proof}

\section{A counterexample to $s = -p^*$} \label{sec:counterexample1}

In this section, we define a sequence $(\overline{X}_n, n)_{n \geqslant 1}$ for which $s = \overline{s} \neq -p^*$. The requirements on the space are quite precise: the topology must be rich enough so that their is no (hidden) convex-tension; the $\sigma$-algebra must be rich enough, so that the space be a l.m.v.s., but not to big, so that the addition be measurable. We will work in the space $\ell^\infty$ of bounded real sequences, equipped with the topology $\tau_{\norme{\cdot}}$ of the norm. The open ball of radius $r$ and centered in $x$ will be denoted by $B(x,r)$. We denote by $\mathcal{C}(\ell^\infty, (\ell^\infty)^*)$ the \emph{cylinder $\sigma$-algebra}, i.e.\ the $\sigma$-algebra over $\ell^\infty$ generated by all the continuous linear functionals on $\ell^\infty$ (see \cite{Vak87}).

\begin{proposition}
$(\ell^\infty, \tau_{\norme{\cdot}}, \mathcal{C}(\ell^\infty, (\ell^\infty)^*))$ satisfies \ref{lm1}, \ref{lm2}, \ref{lm3}, and \ref{lm4}.
\end{proposition}

\begin{proof}
By construction, \ref{lm2} and \ref{lm4} are satistified. To prove \ref{lm1}, it suffices to see that the open balls of $\ell^\infty$ are in $\mathcal{C}(\ell^\infty, (\ell^\infty)^*)$: indeed, for all $x \in \ell^\infty$ and $r > 0$,
\[
B(x,r) = \bigcap_{i=1}^\infty \enstq{y \in \ell^\infty}{\abs{e_i^*(y-x)} < r} ,
\]
where $(e_i^*)_{i \geqslant 1}$ is the family of coordinate linear functionals. Finally, \ref{lm3} holds since we have $\mathcal{C}(\ell^\infty, (\ell^\infty)^*) = \mathcal{F}(\mathcal{C})$, $\mathcal{C}$ being the set of cynlinders $\{ \abs{\lambda(\cdot)} < 1 \}$ for $\lambda \in (\ell^\infty)^*$, and repeating the same proof as that of \Cref{prop:lcss_lmtvs} (see also \cite[Proposition 2.3]{Vak87}).
\end{proof}

Now, we consider the Cantor space $C = \{ 0, 1 \}^\N \subset \ell^\infty$, equipped with the product $\sigma$-algebra $\enspart(\{ 0, 1 \})^{\otimes \N}$, and the measure $\mu = \mathcal{B}(1/2)^{\otimes \N}$. The measure $\mu$ is the pushforward measure of the Lebesgue measure $\lambda$ over the Borel $\sigma$-algebra $\mathcal{B}(\intervalleff{0}{1})$ through the dyadic development $\varphi \colon \intervalleff{0}{1} \to C$.

\begin{proposition}
The Lebesgue measure extends to the $\sigma$-algebra $\enspart(\intervalleff{0}{1})$ if and only if there exists a real-valued measurable cardinal not greater than $2^{\aleph_0}$.
\end{proposition}

\begin{proof}
See \cite[p.\ 131]{Jec06} (see also Fremlin, Real-valued-measurable cardinals (2009) (1D)).
\end{proof}

Suppose that there exists a real-valued measurable cardinal not greater than $2^{\aleph_0}$. Let $\tilde{\lambda}$ be an extension of $\lambda$ to $\enspart(\intervalleff{0}{1})$. We can suppose that $\tilde{\lambda}$ is symmetric, i.e.\ $\sigma \colon x \mapsto 1-x$ is $\tilde{\lambda}$-invariant: take an extension $\tilde{\lambda}$ of $\lambda$ over $\enspart(\intervalleff{0}{1/2})$ and define $\tilde{\lambda}$ over $\enspart(\intervalleof{1/2}{1})$ by $\tilde{\lambda}(A) = \tilde{\lambda}(\sigma(A))$.

Now, the pushforward measure $\tilde{\mu} = \tilde{\lambda} \circ \varphi^{-1}$ is an extension of $\mu$ to $\enspart(C)$. Obviously, $\tilde{\mu}$ extends to a measure over $\enspart(\ell^\infty)$, and we consider the restriction of the latter to $\mathcal{C}(\ell^\infty , (\ell^\infty)^*)$, which will still be denoted by $\tilde{\mu}$. By construction, $\tilde{\mu}$ is symmetric, i.e.\ if $X$ is a random variable with distribution $\tilde{\mu}$, then the distribution of $1-X$ is also $\tilde{\mu}$ (here $1$ denotes the constant sequence $(1)_{n \geqslant 1}$).

In particular, for all $f \in (\ell^\infty)^*$, $f(X)$ has the same distribution as $f(1-X) = f(1) - f(X)$, whence
\[
\forall f \in (\ell^\infty)^* \quad \Espe[f(X)] = f(1/2) .
\]
Now, let $(X_n)_{n \geqslant 1}$ be a sequence of i.i.d.\ random variables on $\ell^\infty$ with distribution $\tilde{\mu}$. Let $(\overline{X}_n)_{n \geqslant 1}$ be the associate sequence of empirical means. Let $p$ be the pressure of $(\overline{X}_n,n)_{n \geqslant 1}$. We have
\begin{align*}
p^*(1/2)
 & = - \inf_{f \in (\ell^\infty)^*} \log \Espe[e^{f(1/2) - f(X)}] \\
 & = - \inf_{f \in (\ell^\infty)^*} \inf_{t \in \R} \log \Espe[e^{t(f(1/2) - f(X))}] \\
 & = 0 .
\end{align*}
Now, let us evaluate the entropy superior $\overline{s}$ of $(\overline{X}_n,n)_{n \geqslant 1}$. For any $x \in \ell^\infty$, we have
\[
\overline{s}(x) = \inf_{\varepsilon > 0} \limsup_{n \to \infty} \frac{1}{n} \sum_{i=1}^\infty \log\Prob\bigl( \abs{X_{1,i} + \dots + X_{n,i} - n x_i} < n \varepsilon \bigr) = -\infty
\]
since $\log \Prob\bigl( \abs{X_{1,i} + \dots + X_{n,i} - n x_i} < n \varepsilon \bigr)$ is negative and independent of $i$. In particular, we have proved that $s = \overline{s} = -\infty \neq -p^*$.

\begin{remark}
We cannot simply consider the probability space $(\ell^\infty, \enspart(\ell^\infty), \tilde{\mu})$: indeed, as we have supposed that there exists a real-valued measurable cardinal not greater than $2^{\aleph_0}$, we can show that the addition on $(\ell^\infty, \enspart(\ell^\infty))$ is not measurable (see \cite{Tal79} and \cite[Lemma 12.2 and Theorem 12.9]{Kun68}).
\end{remark}

\begin{question}
The existence of a real-valued measurable cardinal not greater than $2^{\aleph_0}$ is incompatible with the Continuum Hypothesis. Is there nevertheless a counterexample compatible with the Continuum Hypothesis?
\end{question}

\section{About the equality $p = (-s)^*$} \label{sec:counterexample2}

Throughout this section, $(X_n)_{n \geqslant 1}$ is a sequence of i.i.d.\ random variables on a l.c.m.s.\ $(\mathcal{X}, \tau(\mathcal{C}), \mathcal{F}(\mathcal{C}))$, and $s$ (resp.\ $p$) denotes the entropy (resp.\ pressure) of the sequence $(\overline{X}_n,n)_{n \geqslant 1}$. Notice that our proof of $s = -p^*$ (see \Cref{cramerlcm}) basically relies on \Cref{cramerbanachsep}, so on the dual equality $p = (-s)^*$ valid in any separable Banach space (see \cite[Theorem 12.4]{Petit_2018_CramersTheoremBanach}), where the Fenchel-Legendre transform of a function $g : \mathcal{X} \to [-\infty , +\infty]$ is the function $g^* : \mathcal{X}^* \to [-\infty , +\infty]$ defined by
\[
\forall \lambda \in \mathcal{X}^* \quad g^*(x) = \sup_{x \in \mathcal{X}} \big( \lambda(x) - g(x) \big) .
\]
That being said, the equality $p = (-s)^*$ is generally stronger than the dual equality.

\begin{proposition} \label{prop:petsetoile}
One has $p = (-s)^*$ if and only if $s = -p^*$ and $p$ is $\sigma(\mathcal{X}^*, \mathcal{X})$-lower semi-continuous.
\end{proposition}

\begin{proof}
Just take the convex conjugate and notice that $s$ is always upper semi-continuous (cf. \cite[§5.d]{Mor67}).
\end{proof}

\begin{example} \label{example:pandsstar1}
Here is an example of a l.c.s.s., borrowed from \cite[Section 2]{Dudley_1969_RandomLinearFunctionals}, where $p \neq (-s)^*$ (and $s = -p^*$ by \Cref{cramerlcs}). In particular, it illustrates the fact that, contrary to the equality $s = -p^*$, the equality $p = (-s)^*$ is not preserved by projective limits. Let $\mathcal{Y} = \mathcal{C}(\intervallefo{0}{\omega_1} ; \R)$ be the space of continuous functions on the ordinal space $\omega_1$. Endow $\mathcal{Y}$ with the topology $\tau$ of uniform convergence on compact subsets of $\intervallefo{0}{\omega_1}$. Recall that $\mathcal{Y}^*$ consists of the bounded regular Borel measures on $\intervallefo{0}{\omega_1}$ with compact support.

Each $f \in \mathcal{Y}$ is eventually constant (see, \emph{e.g.}, \cite[p.\ 70]{SteenSeebach_1995_CounterexamplesInTopology}). Denote by $E(f)$ this constant. $E$ is a discontinuous linear functional on $(\mathcal{Y},\tau)$: any compact subset of $\intervallefo{0}{\omega_1}$ is included in some $\intervalleff{0}{\alpha}$ with $\alpha < \omega_1$, so $E$ is not represented by an element of $\mathcal{Y}^*$. Now consider the algebraic dual $\mathcal{Y}^a$ of $\mathcal{Y}$ endowed with the cylinder $\sigma$-algebra $\mathcal{C}(\mathcal{Y}^a,\mathcal{Y})$ and the Dirac mass $\nu$ at $E$. It is shown in \cite[Section 2]{Dudley_1969_RandomLinearFunctionals} that $\nu^*(\mathcal{Y}^*) = 1$, so $\mu \defeq \nu^*$ defines a probability measure on $(\mathcal{Y}^*, \mathcal{C}(\mathcal{Y}^*,\mathcal{Y}))$ (since $\mathcal{C}(\mathcal{Y}^*,\mathcal{Y}) \subset \mathcal{C}(\mathcal{Y}^a,\mathcal{Y})|_{\mathcal{Y}^*}$). And, applying known results on outer measures, one shows that, for any $f \in \mathcal{Y}$, $p(f) = E(f)$. So, $p$ is not $\sigma(\mathcal{Y},\mathcal{Y}^*)$-lower semi-continuous, which provides the announced counterexample in the l.c.s.s.\ $\mathcal{X} = \mathcal{Y}^*$ endowed with the weak-star topology $\sigma(\mathcal{Y}^*,\mathcal{Y})$ and with the cylinder $\sigma$-algebra $\mathcal{C}(\mathcal{Y}^*,\mathcal{Y})$.
\end{example}

\begin{question}
Is there an example where $p \neq (-s)^*$ on a separable normed space?
\end{question}

Now we give a standard sufficient condition for the equality $p = (-s)^*$. A probability measure $\nu$ on $(\mathcal{X}, \tau(\mathcal{C}), \mathcal{F}(\mathcal{C}))$ is \emph{convex tight} if, for all $\alpha > 0$, there exists a measurable relatively compact convex subset $K$ of $\mathcal{X}$ such that $\nu(K) > 1 - \alpha$.

\begin{theorem}
If the distribution of $X_1$ is convex-tight, then $p = (-s)^*$.
\end{theorem}

\begin{proof}
The proof of \cite[Theorem 12.4]{Petit_2018_CramersTheoremBanach} adapts directly as follows. Let $\lambda \in \mathcal{X}^*$ and let $\alpha > 0$. Since the distribution of $X_1$ is convex-tight and using Fatou's lemma, there exists a measurable relatively compact convex subset $K$ of $\mathcal{X}$ such that
\[
\min( p(\lambda) - \alpha, 1/\alpha ) \leqslant \log \mathbb{E}\bigl(e^{\lambda(X_1)} \mathbf{1}_K(X_1) \bigr) \leqslant \inf_{n \geqslant 1} \frac{1}{n} \log \mathbb{E}\big(e^{n\lambda(\overline{X}_n)} \mathbf{1}_K(\overline{X}_n) \big) .
\]
The inequality $p \leqslant (-s)^*$ follows from a version of Varadhan's lemma (see \cite[Lemma 12.2]{Petit_2018_CramersTheoremBanach}, which extends to the present setting).
\end{proof}

\begin{remark}
According to \cref{cramerlcm} and \cref{prop:petsetoile}, an alternate proof consists in showing that $p$ is $\sigma(\mathcal{X}^*,\mathcal{X})$-lower semi-continuous. The following proof is a slight generalization of the one of \cite[Lemma 12.1]{Cer07}. First, denote by $(\tilde{\mathcal{X}}, \tilde{\tau})$ the Hausdorff topological vector space canonically associated to $(\mathcal{X}, \tau(\mathcal{C}))$, where $\tilde{\mathcal{X}} = \mathcal{X}/\overline{\{ 0 \}}$ and $\tilde{\tau}$ is the quotient topology. Also denote by $\pi \colon \mathcal{X} \to \tilde{\mathcal{X}}$ the canonical projection. By Hahn-Banach theorem, $\overline{\{ 0 \}} = \enstq{x \in \mathcal{X}}{\forall \lambda \in \mathcal{X}^* \quad \lambda(x) = 0}$, so the topological dual of $\tilde{\mathcal{X}}$ is isomorph to $\mathcal{X}^*$: we will identify both, writing, for $(\lambda,x) \in \mathcal{X}^* \times \mathcal{X}$, $\lambda(x) = \lambda(\pi(x))$. Moreover, we have $\sigma(\mathcal{X}^*,\mathcal{X}) = \sigma(\mathcal{X}^*,\tilde{\mathcal{X}})$. Since $p$ is convex, it suffices to show that $p$ is $\tau(\mathcal{X}^*,\tilde{\mathcal{X}})$-lower semi-continuous where $\tau(\mathcal{X}^*,\tilde{\mathcal{X}})$ is the Mackey topology, that is to say the locally convex topology admitting for local basis of $0$ the sets $\{ \lambda \in X^* \ |\ \forall \tilde{x} \in \tilde{K} \quad \abs{\lambda(\tilde{x})} < 1 \}$ where $\tilde{K}$ ranges over the absolutely convex $\sigma(\tilde{\mathcal{X}},\mathcal{X}^*)$-compact subsets of $\tilde{\mathcal{X}}$.

Let $\lambda \in \mathcal{X}^*$ and $\alpha > 0$. Since the distribution of $X_1$ is convex-tight, let $K$ be a measurable relatively $\tau(\mathcal{C})$-compact convex subset of $\mathcal{X}$ such that
\[
\log \mathbb{E}\bigl(e^{\lambda(X_1)} \mathbf{1}_K(X_1) \bigr) \geqslant \min(p(\lambda), 1/\alpha ) - \alpha .
\]
The absolutely convex hull $K_1$ of $\overline{K}$ is $\tau(\mathcal{C})$-compact since it is the image of the compact set $\intervalleff{0}{1} \times \overline{K} \times (-\overline{K})$ by the continuous map $(t,x,y) \mapsto (1-t)x + ty$. Finally, $\tilde{K} = \pi(K_1)$ is an absolutely convex  $\tilde{\tau}$-compact (a fortiori $\sigma(\tilde{\mathcal{X}},\mathcal{X}^*)$-compact) subset of $\tilde{\mathcal{X}}$. Considering
\[
V \defeq \enstq{\rho \in \mathcal{X}^*}{\forall \tilde{x} \in \tilde{K} \quad \abs{\rho(\tilde{x}) - \lambda(\tilde{x})} < \alpha} ,
\]
one has, for any $\rho \in V$ and $x \in K$, $\abs{\rho(x) - \lambda(x)} < \alpha$ (because $\pi(K) \subset \tilde{K}$), hence
\[
p(\rho) \geqslant \log \mathbb{E}\bigl(e^{\rho(X_1)} \mathbf{1}_K(X_1) \bigr) \geqslant \min(p(\lambda), 1/\alpha ) - 2 \alpha ,
\]
which yiels the $\tau(\mathcal{X}^*,\tilde{\mathcal{X}})$-lower semi-continuity of $p$.
\end{remark}

However, the convex-tightness of the distribution of $X_1$ is not a necessary condition for the equality $s=-p^*$ as shown in the following example.

\begin{example}
Let $\mathcal{X}$ be the l.c.s.s.\ $\R^\R$ endowed with the product topology $\tau$ and the cylinder $\sigma$-algebra $\mathcal{F}$. There is no convex-tight probability measure on $\mathcal{X}$. Indeed, every compact subset of $\mathcal{X}$ is weakly bounded, whereas no nonempty measurable subset is weakly bounded (any measurable subset being in the $\sigma$-algebra generated by countably many cylinder sets based on finitely many points): so, the only measurable relatively compact subset is the empty set. But, on $\mathcal{X}$, any Dirac mass and even any Gaussian measure has a continuous pressure, of the form $L(\lambda) + Q(\lambda)$ where $L$ is a linear functional on $\mathcal{X}^* = \R^{(\R)}$ (automatically $\sigma(\mathcal{X}^*, \mathcal{X})$-continuous) and $Q$ is a ($\sigma(\mathcal{X}^*, \mathcal{X})$-continuous) nonnegative symmetric bilinear form on $\mathcal{X}^*$ (see, \emph{e.g.}, \cite[Proposition 2.3.9]{Bogachev_1998_GaussianMeasures}).
\end{example}

\bibliographystyle{alpha}
\bibliography{../../../cramer2,../../dg_lin}

\begin{thebibliography}{VTC87}

\bibitem[Bog98]{Bogachev_1998_GaussianMeasures}
Vladimir~I. Bogachev.
\newblock {\em Gaussian measures}.
\newblock Mathematical Surveys and Monographs. AMS, 1998.

\bibitem[Bou66]{BouE}
N.~Bourbaki.
\newblock {\em Théorie des ensembles}.
\newblock Hermann, Paris, 1966.

\bibitem[Bou71]{Bou71}
N.~Bourbaki.
\newblock {\em {Livre III, Topologie g\'en\'erale}}.
\newblock Hermann \& C$^{ie}$, Edition de 1971.

\bibitem[Bou81]{BouEVT}
N.~Bourbaki.
\newblock {\em Espaces vectoriels topologiques}.
\newblock Masson, 1981.

\bibitem[BZ79]{BaZ79}
R.~R. Bahadur and S.~L. Zabell.
\newblock {Large deviations of the sample mean in general vector spaces}.
\newblock {\em Ann. Prob.}, 7(4):587--621, 1979.

\bibitem[Cer07]{Cer07}
R.~Cerf.
\newblock {\em On {C}ram\'er's theory in infinite dimensions}.
\newblock Panoramas et Synth\`eses 23. Soci\'et\'e Math\'ematique de France,
  Paris, 2007.

\bibitem[Cra38]{Cra38}
H.~Cram{\'e}r.
\newblock {Sur un nouveau th\'eor\`eme-limite de la th\'eorie des
  probabilit\'es}.
\newblock {\em Actual. Sci. Indust.}, 736:5--23, 1938.

\bibitem[dA94a]{dAc94a}
A.~de~Acosta.
\newblock {On large deviations of empirical measures in the $\tau$--topology}.
\newblock {\em J. Appl. Probab.}, 31A:41--47, 1994.

\bibitem[dA94b]{dAc94b}
A.~de~Acosta.
\newblock Projective systems in large deviation theory ii: Some applications.
\newblock In J{\o}rgen Hoffmann-J{\o}rgensen, James Kuelbs, and Michael~B.
  Marcus, editors, {\em Probability in {B}anach spaces. 9}, volume~35 of {\em
  Progress in Probability}, pages 241--250. Birkh\"auser Boston, Inc., Boston,
  MA, 1994.
\newblock Papers from the Ninth International Conference held in Sandjberg,
  August 16--21, 1993.

\bibitem[dA97]{DAc97}
Alejandro de~Acosta.
\newblock Exponential tightness and projective systems in large deviation
  theory.
\newblock In {\em Festschrift for {L}ucien {L}e {C}am}, pages 143--156.
  Springer, New York, 1997.

\bibitem[DG87]{DaG87}
D.~A. Dawson and J.~G{ä}rtner.
\newblock {Large deviations from the McKean-Vlasov limit for weakly interacting
  diffusions}.
\newblock {\em Stochastics}, 20:247--308, 1987.

\bibitem[Dra74]{Dra74}
Frank~R. Drake.
\newblock {\em Set theory --- An introduction to large cardinals}, volume~76 of
  {\em Studies in Logic and the Foundations of Mathematics}.
\newblock North-Holland Publishing Co., Amsterdam, 1974.

\bibitem[DS89]{DeS89}
J.-D. Deuschel and D.~W. Stroock.
\newblock {\em {Large Deviations}}.
\newblock Academic Press, 1989.

\bibitem[Dud69]{Dudley_1969_RandomLinearFunctionals}
R.~M. Dudley.
\newblock Random linear functionals.
\newblock {\em Trans. Amer. Math. Soc.}, 136:1--24, 1969.

\bibitem[DZ98]{DeZ93s}
Amir Dembo and Ofer Zeitouni.
\newblock {\em Large deviations techniques and applications}, volume~38 of {\em
  Applications of Mathematics (New York)}.
\newblock Springer-Verlag, New York, second edition, 1998.

\bibitem[FW84]{FrW84}
M.~I. Freidlin and A.~D. Wentzell.
\newblock {\em {Random Perturbations of Dynamical Systems}}.
\newblock Springer--Verlag, 1984.

\bibitem[Jec06]{Jec06}
Thomas Jech.
\newblock {\em Set theory}.
\newblock Springer Monographs in Mathematics. Springer-Verlag, Berlin, 2006.
\newblock The third millennium edition, revised and expanded (corrected 4th
  printing).

\bibitem[Kun68]{Kun68}
Kenneth Kunen.
\newblock {\em Inaccessibility properties of cardinals}.
\newblock ProQuest LLC, Ann Arbor, MI, 1968.
\newblock Thesis (Ph.D.)--Stanford University.

\bibitem[LP95]{LeP95}
J.~T. Lewis and C.-E. Pfister.
\newblock {Thermodynamic probability theory}.
\newblock {\em Russian Math. Surveys}, 50:279--317, 1995.

\bibitem[LPS95]{LPS95}
J.~T. Lewis, C.-E. Pfister, and W.~G. Sullivan.
\newblock {Entropy, concentration of probability and conditional limit
  theorems}.
\newblock {\em Markov Processes Relat. Fields}, 1(3):319--386, 1995.

\bibitem[Mor67]{Mor67}
J.-J. Moreau.
\newblock {\em {Fonctionnelles convexes}}.
\newblock {S\'eminaire sur les Equations aux D\'eriv\'ees Partielles, Coll\`ege
  de France}, 1966--67.

\bibitem[MS48]{Mar48}
E.~Marczewski and R.~Sikorski.
\newblock Measures in non-separable metric spaces.
\newblock {\em Colloquium Math.}, 1:133--139, 1948.

\bibitem[Pet18]{Petit_2018_CramersTheoremBanach}
Pierre Petit.
\newblock Cram\'{e}r's theorem in {B}anach spaces revisited.
\newblock In {\em S\'{e}minaire de {P}robabilit\'{e}s {XLIX}}, volume 2215 of
  {\em Lecture Notes in Math.}, pages 455--473. Springer, Cham, 2018.

\bibitem[San57]{San57}
I.~N. Sanov.
\newblock {\fontencoding{OT2}\selectfont{}O veroяtnosti bolьxih otkloneniй
  sluqaйnыh veliqin}.
\newblock {\em Mat. Sb.}, 42(1):11--44, 1957.
\newblock {English translation: On the probability of large deviations of
  random variables, \emph{Sel. Transl. Math. Statist. Prob. I}: 213--244,
  1961}.

\bibitem[SS95]{SteenSeebach_1995_CounterexamplesInTopology}
Lynn~Arthur Steen and J.~Arthur Seebach.
\newblock {\em Counterexamples in topology}.
\newblock Dover Publications, 1995.

\bibitem[Tal79]{Tal79}
Michel Talagrand.
\newblock Est-ce que {$l^{\infty }$}\ est un espace mesurable\ ?
\newblock {\em Bull. Sci. Math. (2)}, 103(3):255--258, 1979.

\bibitem[Var66]{Var66}
S.~R.~S. Varadhan.
\newblock {Asymptotic probabilities and differential equations}.
\newblock {\em Comm. Pure Applied Math.}, 19:261--286, 1966.

\bibitem[VTC87]{Vak87}
N.~N. Vakhania, V.~I. Tarieladze, and S.~A. Chobanyan.
\newblock {\em Probability distributions on {B}anach spaces}, volume~14 of {\em
  Mathematics and its Applications (Soviet Series)}.
\newblock D. Reidel Publishing Co., Dordrecht, 1987.
\newblock Translated from the Russian and with a preface by Wojbor A.
  Woyczynski.

\bibitem[Z{\u{a}}l02]{Zal02}
C.~Z{\u{a}}linescu.
\newblock {\em Convex analysis in general vector spaces}.
\newblock World Scientific Publishing Co. Inc., River Edge, NJ, 2002.

\end{thebibliography}

\end{document}